\documentclass[11pt, letterpaper]{amsart}

\usepackage[english]{babel}
\usepackage{graphicx,xcolor}
\usepackage[all]{xy}
\usepackage{amsfonts,dsfont,mathrsfs}
\usepackage{amsmath, amsthm, amssymb, mathabx}
\usepackage{mathrsfs}
\usepackage{enumerate, url}
\usepackage[foot]{amsaddr}
\usepackage{fancyhdr}
\usepackage{hyperref}
\usepackage{ifthen,srcltx}
\usepackage{comment}

\newcommand{\Y}{\ensuremath{\Upsilon}}

\newcommand{\G}{\ensuremath{\mathbb{G}}}
\newcommand{\oGL}{\ensuremath{\overline{\Gamma}^L}}
\newcommand{\oGH}{\ensuremath{\overline{\Gamma}^H}}
\newcommand{\oLL}{\ensuremath{\overline{\Lambda}^L}}
\newcommand{\oLH}{\ensuremath{\overline{\Lambda}^H}}

\def\bs{\backslash}

\newcommand{\sslash}{\mathbin{\mkern-3mu /\mkern-5mu/ \mkern-3mu}}

\newcommand{\bbA}{\ensuremath{\mathbb{A}}}
\newcommand{\bbC}{\ensuremath{\mathbb{C}}}

\newcommand{\bbF}{\ensuremath{\mathbb{F}}}
\newcommand{\bbG}{\ensuremath{\mathbb{G}}}

\newcommand{\bbQ}{\ensuremath{\mathbb{Q}}}
\newcommand{\bbR}{\ensuremath{\mathbb{R}}}

\newcommand{\bbZ}{\ensuremath{\mathbb{Z}}}

\newcommand{\calG}{\ensuremath{\mathcal{G}}}

\DeclareMathOperator{\Id}{Id}

\DeclareMathOperator{\Hom}{Hom}
\DeclareMathOperator{\Aut}{Aut}

\DeclareMathOperator{\GL}{GL}

\DeclareMathOperator{\SL}{SL}
\DeclareMathOperator{\PSL}{PSL}
\DeclareMathOperator{\Sp}{Sp}

\DeclareMathOperator{\SO}{SO}

\DeclareMathOperator{\SU}{SU}

\DeclareMathOperator{\diag}{diag}

\DeclareMathOperator{\Tor}{Tor}

\DeclareMathOperator{\Sym}{Sym}
\DeclareMathOperator{\res}{res}

\def\bs{\backslash}

\newcommand\FF{{\mathcal F}}

\newcommand\LL{{\mathcal L}}
\newcommand\MM{{\mathcal M}}

\newcommand\PP{{\mathcal P}}

\newcommand\PMF{{\PP\kern-2pt\MM\FF}}

\newcommand\PML{{\PP\kern-2pt\MM\LL}}

\newcommand{\fsubd}{\mathrel{{\scriptstyle\searrow}\kern-1ex^d\kern0.5ex}}
\newcommand{\bsubd}{\mathrel{{\scriptstyle\swarrow}\kern-1.6ex^d\kern0.8ex}}
\newcommand{\fsubeq}{\mathrel{\raise-.7ex\hbox{$\overset{\searrow}{=}$}}}
\newcommand{\bsubeq}{\mathrel{\raise-.7ex\hbox{$\overset{\swarrow}{=}$}}}

\newcommand{\tsh}[1]{\left\{\kern-.9ex\left\{#1\right\}\kern-.9ex\right\}}

\newtheorem{thm}{Theorem}[section]
\newtheorem{prop}[thm]{Proposition}
\newtheorem{lemma}[thm]{Lemma}
\newtheorem{cor}[thm]{Corollary}
\theoremstyle{definition}
\newtheorem{dfn}[thm]{Definition}

\newtheorem{rem}[thm]{Remark}
\newtheorem{rems}[thm]{Remarks}

\newtheorem{ex}[thm]{Example}

\newtheorem{claim}[thm]{Claim}

\newtheorem{conj}[thm]{Conjecture}

\newtheorem{question}[thm]{Question}
\newtheorem{assume}[thm]{Standing assumptions}

\begin{document}

\title{Commensurators of normal subgroups of lattices}

\author[D. Fisher]{David Fisher}
\address{Department of Mathematics \\ Rice University \\ Houston, TX, USA }

\author[M. Mj]{Mahan Mj}
\address{School of Mathematics \\ Tata Institute of Fundamental Research \\ Mumbai, India}

\author[W. van Limbeek]{Wouter van Limbeek}
\address{Department of Mathematics, Statistics, and Computer Science \\
                 University of Illinois at Chicago \\
                 Chicago, IL, USA}

\date{\today}

\begin{abstract}
We study a question of Greenberg-Shalom concerning arithmeticity of discrete subgroups of semisimple Lie groups with dense commensurators. We answer this question positively for normal subgroups of lattices. This generalizes a result of the second author and T. Koberda for certain normal subgroups of arithmetic lattices in $\SO(n,1)$ and $\SU(n,1)$.
\end{abstract}

\maketitle

\setcounter{tocdepth}{1}
\numberwithin{equation}{section}
\tableofcontents

\section{Introduction}
\label{sec:intro}

\subsection{Main result}\label{sec:mainthm} Let $G$ be a real or $p$-adic semisimple Lie group with finite center and without compact factors, or a finite product of such groups.  More precisely, we consider $G=\G(k)$ where $k$ is a local field of characteristic zero and $\G$ is a semisimple algebraic group defined over $k$ and also products of groups of this type. Let $\Gamma\subseteq G$ be a discrete subgroup with commensurator $\Delta$. For lattices, a landmark theorem of Margulis shows that $\Delta$ detects arithmeticity of $\Gamma$:
\begin{thm}[Margulis] Let $G$ be as above and let $\Gamma$ be an irreducible lattice in $G$. Then $\Gamma$ is arithmetic if and only if $\Delta$ is dense in $G$.
\label{thm:margulis-comm}
\end{thm}
It is natural to speculate that this theorem holds more generally, assuming only that $\Gamma$ is discrete and Zariski dense.  This question was first asked by Greenberg for $G=\SO(n,1)$ in $1974$ and later asked more generally by Shalom \cite{greenberg-comm, mj-comm}.

\begin{question}[{Greenberg-Shalom}] Let $G$ be a semisimple Lie group with finite center and without compact factors. Suppose $\Gamma\subseteq G$ is a discrete, Zariski-dense subgroup of $G$ whose commensurator $\Delta\subseteq G$ is dense. Is $\Gamma$ an arithmetic lattice in $G$?
\label{q:shalom}\end{question}

Greenberg has given a positive answer for finitely generated subgroups of $G=\SL(2,\bbR)$ \cite{greenberg-comm}. Building on work by Leininger-Long-Reid \cite{llr-comm}, Mj gave a positive answer for finitely generated subgroups of $\SL(2,\bbC)$ \cite{mj-comm}.  Koberda-Mj have given a positive answer for normal subgroups $\Gamma$ of lattices in rank $1$ which have infinite abelian quotients \cite{koberda-mj0,koberda-mj}.  We remark that for tree lattices, a counterexample to a result comparable to Koberda-Mj is already implicit in \cite[Proposition 8.1]{BurgerMozes}.  For all other cases of $G$, as well as general infinitely generated subgroups of the above, Question \ref{q:shalom} is open. Greenberg-Shalom's question is closely related to a number of other problems which will be discussed in Section \ref{sec:connect} below.

It is known that any group $\Gamma$ as in Question \ref{q:shalom} has full limit set. This was observed by Greenberg in $G=\SO(n,1)$ and Mj in general \cite{greenberg-comm, mj-comm}. Normal subgroups of lattices are a robust source of groups with full limit set that seem very good candidates for having large commensurators, especially characteristic subgroups.  Our main result is a positive answer to Question \ref{q:shalom} for normal subgroups of lattices:

\begin{thm} Let $G$ be a finite product of almost simple algebraic groups defined over local fields of characteristic zero, and let $\Lambda\subseteq G$ be an irreducible lattice. Suppose $\Gamma\subseteq \Lambda$ is an infinite normal subgroup with dense commensurator $\Delta\subseteq G$. Then $\Gamma$ is an arithmetic lattice and has finite index in $\Lambda$.
\label{thm:main}\end{thm}

We compare this briefly to the results of Koberda-Mj which inspire it.  For their work to apply $\Lambda/\Gamma$ must have
an infinite abelian quotient.  When $G$ and therefore $\Lambda$ have property $(T)$, this never occurs, so our result is new for all normal subgroups of lattices in $G=\Sp(n,1)$ or $F_4^{-20}$.  Furthermore if $\Lambda/\Gamma$ is free (or more generally has a free quotient) it follows from standard constructions of random groups that one can build intermediate normal subgroups $\Gamma \lhd \Gamma' \lhd \Lambda$ such that $\Lambda/\Gamma'$ has property $(T)$ and hence no abelian quotients.  This type of example occurs robustly even in the case $G=\SL(2, \mathbb{R})$ so our results cover a plethora of new examples even when $\Lambda$ is free or a surface group.

Our theorem also implies a suitable version for reducible lattices, see Theorem \ref{thm:reducible}. With a few modifications to the proof, we obtain the same result when $G$ is defined over a local field with positive characteristic when $\Lambda$ is either uniform or arithmetic, see Theorem \ref{thm:char+}. If $G$ has higher rank and $\Lambda$ is irreducible, then $\Gamma$ has finite index in $\Lambda$ by Margulis' Normal Subgroups theorem.  We note here that we do not give a  new proof of that theorem and we use it in our proof.

\subsection{Outline of proofs}\label{sec:outline} We start by proving that $\Delta$ also commensurates $\Lambda$ (Section \ref{sec:arithm}). The rest of the proof is based on the study of a relative profinite completion $L$ of $\Delta$ with respect to its commensurated subgroup $\Gamma$, the similarly defined relative completion $H$ of $\Delta$ with respect to $\Lambda$, and the natural map $L\to H$ between them.

We study the kernel $N$ of this map in several steps. First we quotient by the normal closure $C$ (in $L$) of $\overline{\Gamma}\cap N$, and we show that $L/C\to H$ is (up to finite index) a central extension (Subsection \ref{sec:restrict}). Then we study the (continuous) cohomology class of this extension in $H^2(H,N/C)$, and we show it must be torsion. This readily implies that $N/C$ must be finite (Section \ref{sec:abquot}).

Finally, in Section \ref{sec:end}, we use a direct argument to show that $C$ is compact and since $N/C$ is finite, $N$ is also compact. This readily implies that $\Lambda/\Gamma$ is finite.

In Section \ref{sec:char+}, we discuss the required modifications for the proof in positive characteristic. The proof is largely the same, except where we use finite generation of $\Lambda$. Instead, we use a structural result of Lubotzky that in positive characteristic, a nonuniform lattice is a free product of a free subgroup and lattices in unipotent groups. This description is precise enough that it allows us to circumvent the use of finite generation.

In Section \ref{sec:venky} we establish a technical result on subgroups of $S$-arithmetic lattices, which generalizes a result of Venkataramana (with essentially the same proof). This result is used in the proof of the Main Theorem \ref{thm:main}.

Finally, in Section \ref{sec:reducible}, we use the Main Theorem to deduce its generalization to reducible lattices.

\subsection{Motivations and connections}
\label{sec:connect}

In this section we explain some connections between Greenberg-Shalom's Question \ref{q:shalom} and other well-known problems. In the first subsection we recall a conjecture of Margulis and Zimmer and explain Shalom's original motivation for raising the question.  In the second subsection, we point to a newer set of connections which to the best of our knowledge have not previously appeared in the literature. We only sketch these and leave further details to our forthcoming article with Nic Brody \cite{BFMVL}.

\subsubsection{Margulis-Zimmer Conjecture}
\label{subsec:mz}

A major motivation for Greenberg-Shalom's question is the following conjecture advertised by Margulis-Zimmer in the late '70s, seeking to classify commensurated subgroups of higher rank lattices $\Lambda$. Here we say $\Gamma\subseteq\Lambda$ is \emph{commensurated} if $\Lambda\subseteq \text{Comm}_G(\Gamma)$.

\begin{conj}[{Margulis-Zimmer, see \cite{shalom-willis}}] Let $\G$ be a semisimple algebraic group defined over a global field $k$ and $S$ a finite set of valuations of $k$. Assume $\G$ has higher $S$-rank. Then any commensurated subgroup of the $S$-arithmetic lattice $\Lambda$ is either finite or $S'$-arithmetic for some $S'\subseteq S$.
\label{conj:mz}
\end{conj}

Here, a global field is a number field or a finite extension of a function field $\bbF_p(t)$, and the $S$\emph{-rank} of $\G$ is the sum of the $k_{\nu}$-ranks over all valuations $\nu\in S$, and $\G$ is said to have \emph{higher }$S$\emph{-rank} if its $S$-rank is at least 2.

The connection to Greenberg-Shalom's question occurs as soon as there are at least two valuations, and one of them is non-Archimedean.
To be most transparent, we assume $\Lambda = \SL(n, \bbZ[1/p])$, but this explanation easily generalizes.  Let $\Gamma < \Lambda$ be a commensurated subgroup contained in $\SL(n, \bbZ)$.
Then $\Gamma$ is a discrete subgroup of $\SL(n,\bbR)$ and in the latter group, its commensurator $\Lambda$ is dense. The Margulis-Zimmer conjecture in this case predicts that $\Gamma$ is either finite, or of finite index in $\SL(n,\mathbb{Z})$.

Using well-chosen unipotent generating sets, Venkataramana has proven Conjecture \ref{conj:mz} for arithmetic lattices $\Gamma=\G(\bbZ)$ in simple groups defined over $\bbQ$ \cite{venkataramana-mz}. Shalom-Willis have proved Conjecture \ref{conj:mz} in more instances, including the first that are not simple \cite{shalom-willis}. Their proof crucially relies on fine arithmetic properties for lattices in these groups, namely bounded generation by unipotents. This is the only prior work on Conjecture \ref{conj:mz}. Our results provide some new evidence for the Margulis-Zimmer conjecture.  In particular, it is the first non-trivial result towards the Margulis-Zimmer conjecture that applies to cocompact lattices.  We state a special case to make the contribution clear:

\begin{cor}
\label{corollary:marguliszimmer}
Let $G$ be a simple Lie group of rank 1,  $\Lambda < G$ an arithmetic lattice, $S$ a non-empty finite set of finite places and $\Lambda_S$ the $S$-arithmetic lattice obtained by inverting primes in $S$.  Then any subgroup  $\Gamma< \Lambda$ that is commensurated by $\Lambda_S$ and normalized by $\Lambda$ is finite index in $\Lambda$.
\end{cor}

\noindent The hypothesis that $\Gamma$ is normalized by $\Lambda$ and commensurated by $\Lambda_S$ makes the statement intermediate between the Margulis normal subgroup theorem, where $\Gamma$ is normal in all of $\Lambda_S$, and the Margulis-Zimmer conjecture where no normality is assumed.

\subsubsection{Irreducible subgroups of products}
\label{subsec:products}

In this subsection, we provide some indications of results with Nic Brody that will appear later in \cite{BFMVL}.
The general motivation here is to study discrete subgroups of products that  project indiscretely to the factors.
The only published reference we know  that mentions questions of this type is \cite{FLSS}, though that article
only mentions much weaker questions than the following:

\begin{question}
\label{qtn:irreducible}
Let $G_1$ and $G_2$ be semisimple groups over local fields and $\Lambda < G_1 \times G_2$ a discrete subgroup with both projections dense.  Is $\Lambda$ in fact an irreducible lattice in the product?
\end{question}

At first glance, this question seems overly strong. We state an easy proposition that shows that it is a natural
generalization of Greenberg-Shalom's question.

\begin{prop}
\label{prop:shalomtoirreducible}
Let $G_1$ and $G_2$ be semisimple algebraic groups over local fields with $G_2$ totally disconnected.  Let $\Lambda < G_1 \times G_2$ be discrete with both projections dense.  Then there is a subgroup $\Gamma <\Lambda$ which projects discretely to $G_1$
and $\Gamma$ is commensurated by $\Lambda$.
\end{prop}

The proposition is easy to prove by picking $\Gamma$ to be the intersection of $\Lambda$ with a compact open subgroup
of $G_2$. In fact the proposition really only requires that $G_1$ and $G_2$ be locally compact topological groups with
$G_2$ Hausdorff and totally disconnected (then $G_2$ contains a compact open subgroup).  We note here that in the proof of
Proposition \ref{prop:shalomtoirreducible}, the group $\Gamma$ we construct is most likely infinitely generated and
not normal in any subgroup of $\Lambda$.  So no existing work on Greenberg-Shalom's question \ref{q:shalom} bears directly on Question \ref{qtn:irreducible}.

The difficulty in constructing counterexamples is to guarantee the density of projections.  Of course if either
projection is discrete, $\Lambda$ is best studied as a discrete subgroup of that factor rather than the product.  So it is natural to assume the projections are indiscrete, in which case it is often also easy to prove they are dense.  This will allow
us to see that a number of natural questions would be answered by a positive answer to Greenberg-Shalom's question.  These include

\begin{enumerate}
  \item the nonexistence of surface groups as discrete subgroups of products of rank one $p$-adic groups; this has applications to the arithmetic of three-manifold groups \cite{FLSS};
  \item the nonexistence of discrete free groups with dense projections in many products, answering a question asked by Yves Benoist;
  \item an old question of Lyndon and Ullman on groups generated by parabolics in $\SL(2,\bbR)$, recently made a conjecture by Kim and Koberda \cite{lyndon-ullman, kim-koberda-notfree}.
\end{enumerate}

Detailed explanations of all these connections and some others will appear in the forthcoming paper with Brody \cite{BFMVL}.

\subsection*{Acknowledgments} We thank Michael Larsen, Gopal Prasad, M.S. Raghunathan, Yehuda Shalom and T. N. Venkataramana  for helpful conversations. We particularly thank Thomas Koberda for his involvement in the early phases of this project and Nic Brody for the conversations that lead to \cite{BFMVL}.  DF is supported by NSF grant DMS-1906107.  MM is supported by the Department of Atomic Energy, Government of India, under project no.12-R\&D-TFR-5.01-0500;
and in part by a DST JC Bose Fellowship,  and an endowment from the Infosys Foundation. WvL is supported by NSF DMS-1855371 and DMS-DMS-2203867 and the Max Planck Institute for Mathematics. All three authors were supported by NSF DMS-1928930 while participating in the Fall 2020 semester program hosted by the Mathematical Sciences Research Institute in Berkeley, California.

\section{Proofs}\label{sec:proof}

\subsection{Schlichting completions}\label{sec:schlichting-def} We start by introducing a main tool in our proofs, so-called Schlichting completions. See \cite[Section 3]{shalom-willis} for more details on the facts reviewed here.

\begin{dfn} Let $\Y$ be a countable group and $\Theta<\Y$ a commensurated subgroup. The \emph{Schlichting completion} $\Y\sslash\Theta$ is defined to be the closure of $\Y<\text{Sym}(\Y/\Theta)$ in the topology of pointwise convergence. Here $\text{Sym}(X)$ denotes the set of all bijections $X\to X$.\end{dfn}

Thus $\Y\sslash\Theta$ is a totally disconnected, locally compact, Hausdorff, second countable group. Furthermore, $\Y$ maps onto a dense subgroup of $\Y\sslash\Theta$ with kernel the normal core of $\Theta$ in $\Y$, and the image of $\Theta$ in $\Y\sslash\Theta$ has compact open closure. If in addition $\Y$ is finitely generated, then $\Y\sslash\Theta$ is compactly generated.

\begin{ex} The Schlichting completion of $\Y:=\PSL(n,\bbZ[1/p])$ with respect to $\Theta:=\PSL(n,\bbZ)$ is $\Y\sslash\Theta=\PSL(n,\bbQ_p)$. One may see this by considering the $\PSL(n,\bbQ_p)$-action on its Bruhat-Tits building $X$. There exists $x_0\in X$ (corresponding to the standard order) with stabilizer $\PSL(n,\bbZ_p)$. Since $\Y=\PSL(n,\bbZ[1/p])\subseteq\PSL(n,\bbQ_p)$ is dense and $\Y\cap \PSL(n,\bbZ_p)=\Theta$, we can therefore identify $\Y/\Theta$ with a $\PSL(n,\bbQ_p)$-orbit in $X$. It follows that the closure of $\Y$ in the topology of pointwise convergence on $\Y/\Theta$ is given by $\PSL(n,\bbQ_p)$.
\label{ex:schlichting-arithmetic}
\end{ex}

\subsection{Background regarding algebraic groups} We will need a generalization of the above Example \ref{ex:schlichting-arithmetic} that computes Schlichting completions of general $S$-arithmetic lattices in semisimple Lie groups. To this end, we start by recalling some facts about algebraic groups, especially regarding the image of the universal cover. A general reference is \cite[Sections 1.1 and 1.2]{margulis-book}.

Let $\bbG$ be a connected, algebraic group defined over a field $k$. The subgroup $\bbG(k)^+\subseteq \bbG(k)$ denotes the group generated by unipotent radicals of parabolic $k$-subgroups of $\bbG(k)$.

\begin{rem} If the field $k$ is perfect, then $\bbG(k)^+$ coincides with the group generated by all unipotents of $\bbG(k)$. In particular, this is true if char($k$)=0. However, note that function fields over finite fields are not perfect. If $k=\bbR$, then $\bbG(k)^+$ is the connected component of identity of $\bbG(k)$ (Borel-Tits \cite[Proposition 6.14]{borel-tits-morphismes-abstraits}).\end{rem}
The subgroup $\bbG(k)^+$ is compatible with central isogenies (see also \cite[Proposition I.1.5.5]{margulis-book}):
\begin{prop}[{Borel-Tits \cite[Corollary 6.3]{borel-tits-morphismes-abstraits}}] Let $\pi:\bbG'\to \bbG$ be a central $k$-isogeny of semisimple groups defined over $k$. Then $\pi(\bbG'(k)^+)=\bbG(k)^+$. \label{prop:plus-isogeny}\end{prop}
Clearly $\bbG(k)^+\subseteq \bbG(k)$ is normal. Therefore if $\bbG$ is semisimple and contains parabolic subgroups, we can expect $\bbG(k)^+$ will be very large. Indeed, we have:
\begin{thm}[{Platonov \cite{platonov-ktconj}}] Let $\bbG$ be a connected and simply-connected, semisimple, algebraic group defined over a local field $k$ without $k$-anisotropic factors. Then $\bbG(k)^+=\bbG(k)$.

In particular, if $\bbG$ is not necessarily simply-connected but otherwise satisfies the above assumptions, and $\pi:\widetilde{\bbG}\to\bbG$ denotes its universal cover, then $\bbG(k)^+=\pi(\widetilde{\bbG}(k))$.
	\label{thm:ktconj}
		\end{thm}

\subsection{Schlichting completions of arithmetic lattices}\label{sec:notation-arithmetic} We briefly introduce notation for ($S$-)arithmetic lattices in algebraic groups.  Let $\bbG$ be an algebraic group defined over a global field $k$ (i.e. a number field if char$(k)$=0, or a finite extension of $\bbF_q(t)$ if char$(k)>0$). Let $V$ denote the set of all places of $k$, and let $\bbA$ denote the adeles over $k$. We fix a $k$-rational embedding $\bbG\hookrightarrow \GL(N)$ for some $N$, and use this to define $\bbG(\bbA):=\Pi'_{v\in V} \bbG(k_v)$ as the restricted product of all (inequivalent) completions of $k$. For a finite set of places $S\subseteq V$, likewise define $\bbG(\bbA_S):=\Pi'_{v\notin S} \bbG(k_v)$ as the restricted product away from $S$. Also set $G_S:=\prod_{s\in S} \bbG(k_s)$.

Let $S\subseteq V$ be a nonempty finite set of places (containing all archimedean places if char$(k)=0$), and let $U\subseteq \bbG(\bbA_S)$ be a compact open subgroup. Then $\Lambda_{S,U}:=\bbG(k)\cap U$ projects to a lattice in $G_S$. Note that the commensurability class of $\Lambda_{S,U}$ is independent of $U$. Any subgroup $\Lambda$ of $G_S$ commensurable to $\Lambda_{S,U}$ is called an $S$-\emph{arithmetic} lattice. If $S$ consists precisely of all the archimedean valuations of $k$, we simply say $\Lambda$ is \emph{arithmetic}.

The goal of the rest of this subsection is the following generalization of Example \ref{ex:schlichting-arithmetic} that computes the Schlichting completion of an $S$-arithmetic lattice with respect to a $T$-arithmetic sublattice (where $T\subsetneq S$):

\begin{prop} Let $\bbG$ be a connected, almost $k$-simple, adjoint algebraic group defined over a global field $k$. Let $T\subsetneq S$ be finite sets of places of $k$ (containing all archimedean places if char($k$)=0) such that there exists at least one place in $T$ and one place in $S\bs T$ at which $\bbG$ is noncompact. Fix a compact open subgroup
	$$U_T=U_{S\bs T} \times U_S\subseteq \bbG(\bbA_T)=G_{S\bs T}\times \bbG(\bbA_S).$$
Set $\Y:=\Lambda_{S,U_S}$ and $\Theta:=\Lambda_{T,U_T}$.

Then $\Y$ commensurates $\Theta$ and $\Y\sslash\Theta$ is (isomorphic to) the closure of $\Y$ in $G_{S\bs T}^{\text{is}}:=G_{S\bs T}/G_{S\bs T}^{\text{an}}$ (the quotient of $G_{S\bs T}$ by the product of all its anisotropic factors $G_{S\bs T}^{\text{an}}$), and the isomorphism $\Y\sslash\Theta\to\overline{\Y}$ restricts to identity on $\Y$. Further, $\overline{\Y}\subseteq G_{S\bs T}^{\text{is}}$ is closed, normal and cocompact, and the compact quotient $G_{S\bs T}^{\text{is}}/\overline{\Y}$ is abelian and has bounded exponent. If char$(k)=0$, the quotient is finite.
\label{prop:schlichting-arithmetic}\end{prop}
\begin{rem} The assumption that there exists a place in $T$ at which $\bbG$ is noncompact guarantees that $\Theta$ is infinite. Likewise, the existence of such a place in $S\bs T$ implies that $\Theta\subseteq \Y$ has infinite index.\label{rem:infty}\end{rem}
\begin{proof} Since $G_{S\bs T}\times U_S$ commensurates its compact open subgroup $U_T$, and the relation of commensurating a subgroup is preserved under taking intersections, we have that $\Y=\bbG(k)\cap (G_{S\bs T}\times U_S)$ commensurates $\Theta=\bbG(k)\cap U_T$. Next, we recall a few general results about $S$-arithmetic lattices that when combined show that the closure of $\Y$ in $G_{S\bs T}^{\text{is}}$ is cocompact (finite index if char$(k)=0$), normal, with quotient that is abelian with bounded exponent.

Let $\pi_{S\bs T}:\widetilde{G}_{S\bs T}\to G_{S\bs T}$ be the universal cover, and likewise define $\pi_{S\bs T}^{\text{is}}$. Recall that the image of $\pi_{S\bs T}^{\text{is}}$ is closed, normal, and cocompact and the quotient $G_{S\bs T}^{\text{is}}/\pi_{S\bs T}^{\text{is}}(\widetilde{G}_{S\bs T}^{\text{is}})$ is abelian with bounded exponent. If char$(k)=0$, then the quotient is finite. These properties are due to Borel-Tits \cite[3.19-20 and 6.3]{borel-tits-morphismes-abstraits} for $(G_{S\bs T}^{\text{is}})^+$ instead of $\pi_{S\bs T}^{\text{is}}(\widetilde{G}_{S\bs T}^{\text{is}})$, and, since $G_{S\bs T}^{\text{is}}$ does not have any anisotropic factors,  $(G_{S\bs T}^{\text{is}})^+ = \pi(\widetilde{G}_{S\bs T}^{\text{is}})$ by Platonov's Theorem \ref{thm:ktconj}.


Therefore it suffices to show that $\overline{\Y}$ (closure taken in $G_{S\bs T}^{\text{is}}$) contains $(G_{S\bs T}^{\text{is}})^+=\pi_{S\bs T}^{\text{is}}(\widetilde{G}_{S\bs T}^{\text{is}})$. Set $\widetilde{\Y}:=\pi_{S\bs T}^{-1}(\Y)$. Then $\widetilde{\Y}$ is commensurable with the $S$-integers in $\widetilde{\bbG}(k)$ \cite[I.3.2.9]{margulis-book}, which are dense in $\widetilde{G}_{S\bs T}$ by strong approximation (due to Platonov in characteristic zero \cite{platonov-strongapprox} and Prasad in positive characteristic \cite{prasad-strongapprox}; these references state strong approximation in terms of the closure of $k$-rational points in the adeles, but see \cite[Proposition 7.2(2)]{platonov-rapinchuk-book} for a reformulation in terms of $S$-integers). We note here that strong approximation applies because $\widetilde{\bbG}$ is simply-connected and $\widetilde{G}_S$ is noncompact.

It follows that the closure of $\pi_{S\bs T}^{-1}(\Y)$ has finite index in $\widetilde{G}_{S\bs T}$. But $\widetilde{G}_{S\bs T}^+$ does not have proper subgroups of finite index (Borel-Tits \cite[6.7]{borel-tits-morphismes-abstraits}) so the closure of $\pi_{S\bs T}^{-1}(\Y)$ contains $\widetilde{G}_{S\bs T}^+$, and hence the closure of $\Y\subseteq G_{S\bs T}^{\text{is}}$ contains $(G_{S\bs T}^{\text{is}})^+$, as desired.

It remains to compute that the Schlichting completion $\Y\sslash\Theta$ is given by the closure $\overline{\Y}$ of $\Y$ in $G_{S\bs T}^{\text{is}}$. For this, we need to find a faithful action of $\Y$ on a discrete set whose point-stabilizers are (commensurable to) $\Theta$. We use the action on the Bruhat-Tits building $X_{S\bs T}$ of $G_{S\bs T}^{\text{is}}$. Since $G_{S\bs T}^{\text{is}}$ acts continuously and properly on $X_{S\bs T}$, its point-stabilizers are compact and open, and hence commensurable to any compact open subgroup. In particular, the point-stabilizers are commensurable to the image $U_{S\bs T}^{\text{is}}$ of $U_{S\bs T}\subseteq G_{S\bs T}$ in $G_{S\bs T}^{\text{is}}$. It follows that the point-stabilizers in $\Y$ are commensurable to $\Y\cap U_{S\bs T}^{\text{is}}$, which has finite index in (and hence is commensurable to) $\Theta$. Fix any vertex $x_0$ and let $\Sigma$ be the point-stabilizer (in $\Y$) of this vertex. Since $\Sigma$ and $\Theta$ are commensurable subgroups of $\Y$ and Schlichting completions only depend on the commensurability class of the subgroup, we have $\Y\sslash\Theta\cong \Y\sslash\Sigma$, and the isomorphism restricts to identity on $\Y$. We will now show that $\Y\sslash\Sigma$ is (isomorphic to) the closure of $\overline{\Y}\subseteq G_{S\bs T}^{\text{is}}$ via an isomorphism that restricts to identity on $\Y$.

 The action of $G_{S\bs T}^{\text{is}}$ on $X_{S\bs T}$ gives a continuous map $\varphi:G_{S\bs T}^{\text{is}}\to \Sym(X_{S\bs T})$. The orbit $\Y x_0$, which we can identify with $\Y/\Sigma$, is invariant under the closure $\overline{\Y}$ of $\Y$ in $G_{S\bs T}^{\text{is}}$, so that $\varphi$ restricts to a continuous map $\overline{\Y}\to \Sym(\Y/\Sigma)$. By continuity, the image is contained in the closure of $\Y\subseteq\Sym(\Y/\Sigma)$, which is $\Y\sslash\Sigma$. Hence we have a continuous surjective map $\overline{\Y}\to\Y\sslash\Sigma$ which restricts to identity on $\Y$. It remains to show this is a homeomorphism.

First we note that it is injective: Its kernel is a normal subgroup of $\overline{\Y}$ and hence is normalized by $(G_{S\bs T}^{\text{is}})^+$. By a result of Tits \cite{tits-simple}, any subgroup normalized by $(G_{S\bs T}^{\text{is}})^+$ is either central in $G_{S\bs T}^{\text{is}}$ or contains $(\widetilde{G}_{S\bs T}^{\text{is}})^+$. The latter option is impossible: Then the kernel would be cocompact, so that $\Sigma\subseteq\Y$ would be finite index, contradicting that $\Sigma$ is $T$-arithmetic and $\Y$ is $S$-arithmetic and existence of a place in $G_{S\bs T}$ at which $\bbG$ is noncompact. Hence $\overline{\Y}\to\Y\sslash\Sigma$ has kernel contained in the center of $G_{S\bs T}^{\text{is}}$. However, since $\bbG$ is adjoint, the center of $G_{S\bs T}^{\text{is}}$ is trivial, so that $\overline{\Y}\to\Y\sslash\Sigma$ is injective. Finally, since $\overline{\Y}$ acts properly on $X_{S\bs T}$, the map $\overline{\Y}\to \Sym(\Y / \Sigma)$ is not just continuous and injective, but a homeomorphism onto its image. \end{proof}

\subsection{Cayley-Abels graphs}\label{sec:ca-graphs} In the case that $\Y$ is finitely generated, $\Y\sslash\Theta$ can be realized as a group of automorphisms of a locally finite graph, namely its Cayley-Abels graph (with respect to a compact generating set). Its construction is due to Abels \cite{abels-cayley}. Our discussion follows \cite[11.3]{monod-book}. Let $H$ be a totally disconnected, compactly generated Hausdorff topological group. Then $H$ has a compact open subgroup $K$. We fix a bi-$K$-invariant compact symmetric generating set $C$. Then the Cayley-Abels graph $\calG=\calG(H,C, K)$ of $H$ with respect to $C$ and $K$ is defined to have vertex set $H/K$ and edges between $hK$ and $hcK$ for all $h\in H$ and $c\in C$. By compactness of $C$, this graph is $d$-regular for $d=|C/K|<\infty$, and since $C$ generates $H$, it is connected. Further, $H$ acts on $\calG$ by graph automorphisms with vertex stabilizers given by conjugates of $K$. In particular, this action is isometric and hence proper.

If $\Y$ is a finitely generated group with commensurated subgroup $\Theta$, then $K:=\overline{\Theta}^H$ is a compact open subgroup. With this choice, the Cayley-Abels graph has vertex set $H/K=\Y/\Theta$.

\subsection{Arithmeticity of $\Lambda$}\label{sec:arithm} As remarked in the introduction, since $\Gamma$ is normal in $\Lambda$ and $\Lambda$ is an irreducible lattice in $G$, Margulis' Normal Subgroups Theorem allows us to assume that $G$ is a rank one simple group. Further, we can pass to $G/Z(G)$ and assume that $G$ has trivial center. We make these assumptions without further comment for the rest of the paper. We start by proving:
\begin{prop} Under the assumptions of Theorem \ref{thm:main}, $\Delta$ commensurates $\Lambda$. In particular, $\Lambda$ is $T$-arithmetic for some finite set of places $T$ of $k$.
\label{prop:Lambda-arithm}
\end{prop}
\begin{proof} The second part follows from the first by Margulis' commensurator rigidity theorem. For the first claim, let $\delta\in\Delta$. We claim that $\Lambda$ normalizes a finite index subgroup of $\Gamma\cap\Gamma^\delta$. This is clear if $\Gamma$ is finitely generated (so that it only admits finitely many subgroups of a given index), but in general we argue as follows:

The group $\Lambda_\delta:=\langle \Lambda,\delta\rangle$ is finitely generated and has commensurated subgroup $\Gamma$. Fix a Cayley-Abels graph of $\Lambda_\delta\sslash\Gamma$ with vertex set $\Lambda_\delta/\Gamma$. Then $\Gamma\cap\Gamma^\delta$ is the pointwise stabilizer in $\Lambda_\delta$ of $\{e\Lambda_\delta,\delta\Lambda_\delta\}\subseteq \Lambda_\delta/\Gamma$. Let $R$ be the distance between these points with respect to the graph metric on the Cayley-Abels graph, and let $\Gamma'$ be the pointwise stabilizer of the ball of radius $R$ centered at $e\Lambda_\delta\in \Lambda_\delta/\Gamma$.

Note that the stabilizer (in $\Delta$) of a point in $\Delta/\Gamma$ only depends on its image in $\Delta/\Lambda$: Indeed, if $\delta\in\Delta$ and $\lambda\in\Lambda$, then the stabilizer of $\delta\Gamma$ is just $\Gamma^\delta$, whereas the stabilizer of $\delta\lambda\Gamma$ is $\Gamma^{\delta\lambda}=\Gamma^\delta$. Therefore $\Gamma'$ is just the pointwise stabilizer of all fibers that are within distance $R$ of $\Lambda/\Gamma$ (which is the fiber over $e\Lambda\in\Lambda_\delta/\Lambda$), and since this set of fibers is $\Lambda$-invariant, $\Lambda$ normalizes $\Gamma'$.

Since $\Lambda$ normalizes $\Gamma$ and $\Gamma'$, we can consider the conjugation action $\Lambda\to\Aut(\Gamma/\Gamma')$. Since $\Aut(\Gamma/\Gamma')$ is finite, a finite index subgroup $\Lambda'\subseteq\Lambda$ leaves $(\Gamma\cap\Gamma^\delta)/\Gamma'$ invariant, so that $\Lambda'$ normalizes $\Gamma\cap\Gamma^\delta$. Likewise, there is a finite index subgroup $\Lambda''\subseteq \Lambda^\delta$ that normalizes $\Gamma\cap\Gamma^\delta$. However, since $\Gamma\cap\Gamma^\delta$ is a discrete Zariski-dense subgroup of $G$, its normalizer $N_G(\Gamma\cap\Gamma^\delta)$ is also discrete, and hence so is its subgroup $\langle \Lambda',\Lambda'' \rangle$. Since $\langle \Lambda', \Lambda''\rangle$ contains a lattice and is discrete, it is itself a lattice. Therefore $\Lambda'$ and $\Lambda''$ are both finite index in $\langle \Lambda',\Lambda''\rangle$ and hence are commensurable.\end{proof}
$T$-arithmeticity of $\Lambda$ means there exists a a connected, semisimple, algebraic group $\bbG$ defined over a number field $k$, and a finite set of places $T$ of $k$ (containing all archimedean places) such that $\Lambda$ is $T$-arithmetic in $\bbG(k)$. Since $G$ has trivial center, $\bbG$ is adjoint. Since $\Lambda$ is an irreducible lattice, $\bbG$ is $k$-almost simple.

$T$-Arithmeticity of $\Lambda$ lets us reduce to the case where $\Delta$ is an $S$-arithmetic lattice by way of the following generalization of a result of Venkataramana (see \cite[Proposition 2.3]{lubotzky-zimmer}):
\begin{lemma} Let $\Theta\subseteq \bbG(k)$ be a subgroup containing $\Lambda$ whose projection to $\bbG(k_s)$ is bounded for almost all places $s$, and let $S$ be the (finite) set of places where the projection of $\Theta$ is unbounded. Then $\Theta$ is $S$-arithmetic.
\label{lem:gen-venky}\end{lemma}
Venkataramana's result is stated in \cite{lubotzky-zimmer} in characteristic zero and under the assumption that $S$ contains an archimedean place at which $\bbG$ is noncompact. The proof of the above lemma is rather technical and we postpone it until Section \ref{sec:venky}.

\begin{rem} By the above lemma, any finitely generated subgroup of $\Delta$ containing $\Lambda$ is $S$-arithmetic for some finite set of places $S$. As soon as such a subgroup has unbounded projection to a place not in $T$, we have $T\subsetneq S$. Such a place exists because otherwise any finitely generated subgroup of $\Delta$ has bounded projection to all places outside of $T$ and hence would be $T$-arithmetic. Hence $\Lambda$ would have finite index in any finitely generated subgroup of $\Delta$ containing it, which contradicts that $\Lambda\subseteq G$ is discrete and $\Delta\subseteq G$ is dense.

We will now replace $\Delta$ by a finitely generated $S$-arithmetic subgroup containing $\Lambda$, so that henceforth $\Delta$ is $S$-arithmetic for some finite set of places  $S\supsetneq T$. Note that there exists a place in $T$ at which $\bbG$ is noncompact because $\bbG$ has $T$-rank 1, and $\bbG$ is noncompact at every place in $S\bs T$ because $\Delta$ has unbounded projections at those places. In particular, we are in the situation of Proposition \ref{prop:schlichting-arithmetic}, and hence $\Delta\sslash\Lambda$ is given by the closure of $\Delta$ in $G_{S\bs T}^{\text{is}}$, and is a normal, finite index subgroup of $G_{S\bs T}^{\text{is}}$.
\label{rem:one-prime}\end{rem}

\subsection{Maps of descent and restriction} \label{sec:restrict} Let $L:=\Delta \sslash \Gamma$ be the completion of $\Delta$ with respect to $\Gamma$ (i.e. the closure in the topology of pointwise convergence of $\Delta$ acting on $\Delta / \Gamma$). Likewise let $H:=\Delta\sslash\Lambda$ be the completion of $\Delta$ with respect to $\Lambda$. Note that $L$ and $H$ are locally compact, second countable groups. Further $\oGL$ (the closure of $\Gamma$ in $L$) and $\oLH$ are compact open subgroups (since they are isotropy groups of points in the coset spaces).

The construction of Schlichting completions is functorial with respect to the inclusion $\Gamma\hookrightarrow \Lambda$:

\begin{prop} There is a continuous surjection $q:L\to H$ such that the canonical projection $\pi: \Delta / \Gamma\to \Delta / \Lambda$ is $q$-equivariant.
\label{prop:LtoH}\end{prop}
\begin{proof} It is clear that $\pi$ is $\Delta$-equivariant, so $\Delta$ permutes the fibers of $\pi$. We claim that the same is true for $L$. Indeed, suppose $g\in L$ and choose $g_n\in\Delta$ such that $g_n\to g$ pointwise on $\Delta/\Gamma$. Let $x,y\in\Delta/\Gamma$ be points that lie in the same fiber of $\pi$. Then for any $n\geq 1$, their translates $g_n x$ and $g_n y$ also lie in the same fiber of $\pi$. On the other hand, for $n\gg1$, we have that $g x = g_n x$ and $g y = g_n y$. Hence $gx$ and $gy$ belong to the same fiber of $\pi$.

Therefore every $g\in L$ descends to a map $q(g)\in \text{Sym}(\Delta/\Lambda)$. It is obvious that if $g_n\to g$ pointwise on $\Delta/\Gamma$, then $q(g_n)\to q(g)$ on $\Delta/\Lambda$, and hence $q$ is continuous. Since $\Delta\subseteq L$ is dense and $q(\Delta)\subseteq H$, we conclude that $q:L\to H$ is a continuous surjection.\end{proof}

\begin{prop} Set $N:=\ker(q)$. Then $N\subseteq \oLL$.
\label{prop:NinGamma}\end{prop}
\begin{proof} Consider the $N$-orbit of the basepoint $e\Gamma\in \Delta / \Gamma$. Since $N$ is the kernel of $q$ and $\pi:\Delta/\Gamma\to\Delta/\Lambda$ is $q$-equivariant, we see that $N\cdot e\Gamma\subseteq q^{-1}(e\Lambda)=\Lambda / \Gamma$. Since $\oGL$ is the isotropy group of $e\Gamma\in \Delta / \Gamma$, we find that $N\subseteq \Lambda \oGL$ (the product on the right-hand side is taken as sets). The right-hand side is clearly inside $\oLL$. \end{proof}
Since $N$ is normal in $\overline{\Lambda}^L$, the product $N\oGL$ (as sets) is a subgroup of $\oLL$. We have:
\begin{prop} $N\oGL$ has finite index in $\oLL$.
\label{prop:propbyN}\end{prop}
\begin{proof} We start by observing the map $q:L\to H$ is open: Indeed, since $\Delta$ is countable and $L$ is locally compact, $L$ is also $\sigma$-compact, and a classical result of Freudenthal is that any continuous surjective homomorphism between Hausdorff locally compact groups with $\sigma$-compact domain, is open \cite[Cor. 2.D.6]{cornulier-dlh-book}. Therefore $\oGH=q(\oGL)$ is an open subgroup of the compact group $\oLH$, and hence has finite index. Therefore $N\oGL=q^{-1}(\oGH)$ has finite index in $q^{-1}(\oLH)=\oLL$.\end{proof}

By replacing $\Lambda$ with its finite index subgroup $\Lambda\cap\oGH$ and $\Delta$ by the normal closure of $\Lambda$ , we can arrange that:
\begin{assume} $\oLH=\oGH$ and hence we have equality in the above proposition, i.e. $\oLL=N\oGL$.
\label{assume:oLH=oGH}\end{assume}

Note that $\Lambda / \Gamma = q^{-1}(e\Lambda)$ is $\oLL$-invariant (inside $\Delta / \Gamma$). Let $\res:\oLL\to \text{Sym}(\Lambda/\Gamma)$ denote the restriction map.
\begin{prop} $\res$ is continuous (with respect to the topology of pointwise convergence), has image $\Lambda/\Gamma$ (acting by translations) and kernel $\oGL$.
\label{prop:res}
\end{prop}
\begin{proof} Clearly if a sequence of maps converge pointwise on $\Delta/\Gamma$, they do so on the subset $\Lambda/\Gamma$. Therefore $\res$ is continuous. Since $\Lambda$ is dense in $\oLL$ and $\res(\Lambda)=\Lambda/\Gamma$ is closed in Sym($\Lambda/\Gamma)$, the image of $\res$ is exactly $\Lambda/\Gamma$. It remains to determine $\ker(\res)$: Clearly $\oGL\subseteq \ker(\res)$. Conversely, suppose $g\in\oLL$ satisfies $\res(g)=e$. Write $g=\lim_n \lambda_n$ for some $\lambda_n\in\Lambda$. Then $\lambda_n$ fixes the basepoint $e\Gamma\in\Delta/\Gamma$ for $n\gg1$, and hence $\lambda_n\in\Gamma$ for $n\gg1$.\end{proof}
Since $\Gamma$ is normal in $\Lambda$, we also have that $\oGL$ is normal in $\oLL$. In particular, $\oGL\cap N$ is a compact open normal subgroup of $N$.

\begin{claim} $\res$ restricts to an isomorphism $N/(\oGL\cap N)\overset{\cong}{\longrightarrow} \Lambda/\Gamma$.
\label{claim:res-res-iso} \end{claim}
\begin{proof} Recall from Proposition \ref{prop:res} that
	$$\res:\oLL\to \Lambda/\Gamma$$
is surjective and $\ker(\res)=\oGL$. So $\res$ descends to an isomorphism
	$$\oLL/\oGL\overset{\cong}{\to} \Lambda/\Gamma.$$
On the other hand, by the modification of $\Lambda$ in Standing assumption \ref{assume:oLH=oGH}, we have $\oLL=N\oGL$, so that
	$$\oLL/\oGL = (N\oGL)/\oGL \cong N/(\oGL\cap N).$$
\end{proof}

\subsection{Down towards a central extension}\label{sec:down}
Write $C:=\ll \hspace{-0.1 cm} \oGL\cap N\hspace{-0.1 cm}\gg_L$ for the normal closure in $L$ of $\oGL\cap N$. Note that $C$ is a normal subgroup of $L$ contained in $N$, and is open in $N$ (because $\oGL\cap N$ is open in $N$), hence also closed.

Now we divide by $C$ and study the extension
	$$1\to N/C \to L/C \to H\to 1.$$
Restriction of this short exact sequence to $\oLL/C$ (as middle term) gives
		\begin{equation}
		1\to N/C \to \oLL/C\to \overline{\Lambda}^{H}\to 1.
		\label{eq:seq-res}
			\end{equation}
\begin{claim} The short exact sequence \eqref{eq:seq-res} splits trivially as an extension of topological groups.
\label{claim:l-split}\end{claim}
\begin{proof} We provide a left splitting. First note that by the above claim, $\res$ descends to a map
		$$\res/C : \oLL/C\to (\Lambda/\Gamma)/\res(C)$$
that restricts to an isomorphism $N/C\to (\Lambda/\Gamma)/\res(C)$. Then the desired left splitting is simply the composition of $\res/C$ with the inverse of its restriction to $N/C$:
	$$(\res/C)|_{N/C}^{-1} \circ (\res/C): \oLL/C\longrightarrow   N/C.$$
It is clear this map restricts to identity on $N/C$. This is a splitting of topological groups because both $\res/C$ and its inverse are continuous.\end{proof}
Hence we have a closed subgroup $K\subseteq \oLL/C$ (namely the kernel of the above left splitting) such that
	$$\oLL/C\cong (N/C) \times K$$
(and this is compatible with the sequence, i.e. the copy of $N/C$ is the one given by the inclusion in the sequence, and $K$ projects isomorphically to $\overline{\Lambda}^{H}$.)

We note for use below:
\begin{claim} $\Gamma\subseteq K$. \label{claim:Gamma-subset-K}\end{claim}
\begin{proof} $K$ is precisely the kernel of the left-splitting produced in the proof of Claim \ref{claim:l-split}, i.e. the kernel of
	$$(\res/C)|_{N/C}^{-1} \circ (\res/C): \oLL/C\longrightarrow   N/C.$$
The kernel of this map coincides with $\ker(\res/C)=(C\oGL)/C$ and therefore contains (the image of) $\Gamma$.\end{proof}
Now consider the normal closure $\ll \hspace{-0.1 cm} K\gg_{L/C}$.
\begin{claim} Under the map $L/C\to H$, the image of $\ll \hspace{-0.1 cm} K\gg_{L/C}$ has finite index.\end{claim}
\begin{proof} Under $q:L\to H$, the image of $K$ is $\oLH$ and therefore the image of $\ll \hspace{-0.1 cm} K\gg_{L/C}$ is $\ll \hspace{-0.1 cm} \oLH\gg_H$. Taking instead the normal closure $\ll \hspace{-0.1 cm}\Lambda\gg_\Delta$ in $\Delta$ and noting that by Lemma \ref{lem:gen-venky}, $\Delta$ is $S$-arithmetic and in particular satisfies the Margulis Normal Subgroups Theorem, we see that $\ll \hspace{-0.1 cm} \Lambda\gg_\Delta$ has finite index in $\Delta$. Since $\Delta$ is dense in $H$ and $\ll \hspace{-0.1 cm} \oLH\gg_H\subseteq H$ is open (and hence closed), it follows that $\ll \hspace{-0.1 cm} \oLH\gg_H$ has finite index in $H$.\end{proof}
Write $H'$ for the image of $\ll \hspace{-0.1 cm} K\gg_{L/C}$ in $H$, and let $L'$ be the preimage of $H'$ in $L/C$, so that we can write
	\begin{equation}
	1\to N/C\to L'\to H'\to 1.
	\label{eq:ses-mod-c}
	\end{equation}
Since $K$ commutes with $N/C$, so does its normal closure. In particular $\ll \hspace{-0.1 cm} K\gg_{L/C}$ intersects $N/C$ in a central subgroup $Z$ that is normal in $L/C$. Now we quotient by $Z$ and find the sequence
	\begin{equation}
	1\to (N/C)/Z \to L'/Z \to H'\to 1
	\label{eq:ses-mod-cz}
	\end{equation}
	
which splits trivially as
	$$L'/Z\cong (N/C)/Z \times (\ll \hspace{-0.1 cm} K\gg_{L/C} / Z).$$
\begin{claim} $Z$ has finite index in $N/C$.
\label{claim:z-nc-fi}\end{claim}
\begin{proof} We have a natural map $\Delta\to (L/C)/Z$. Let $\Delta'\subseteq \Delta$ consist of those elements whose projection to $L/C$ lies in $L'$. So $\Delta'$ is a finite index subgroup of $\Delta$. Composition of $\Delta'\to L'\to L'/Z$ with the projection from
	$$L'/Z\cong (N/C)/Z \times (\ll \hspace{-0.1 cm} K\gg_L / Z)$$
to its first factor restricts is a map $\rho:\Delta'\to (N/C)/Z$. Then $\rho$ is surjective: Indeed, the image of $\Delta$ in $L'$ is dense and $N/C$ is discrete. Furthermore, $\rho$ has infinite kernel: Indeed $\ker(\rho)=\Delta\cap (\ll \hspace{-0.1 cm} K\gg_L/Z)$. Since $\Gamma\subseteq K$ by Claim \ref{claim:Gamma-subset-K}, $\rho$ is trivial on $\Gamma$.

Since $\Delta'$ is $S$-arithmetic and $\rho$ has infinite kernel, by Margulis' normal subgroups theorem, its image is finite. Since the image of $\rho$ is isomorphic to $(N/C)/Z$, we conclude that $Z$ has finite index in $N/C$.\end{proof}

\subsection{Splitting of the central extension} \label{sec:abquot}
The goal of this section is to prove:
\begin{lemma} \label{lemma:c-fi} Assume the notation of the Main Theorem \ref{thm:main}, and define $C$ as above. Then $C\subseteq N$ has finite index.\end{lemma}

\begin{proof} Since $Z\subseteq N/C$ is finite index, it suffices to prove that $Z$ is finite. Recall that $L'$ is a suitable finite index subgroup of $L$.  Consider the short exact sequence
	\begin{equation}
	1\to Z\to L'\to L'/Z\to 1.
	\label{eq:ses-z}
	\end{equation}
Write $H'':=L'/Z$ and recall that $H''\cong (N/C)/Z\times H'$. The above short exact sequence is a central extension, and hence is determined by a cohomology class in $H^2(H'',Z)$. We claim this cohomology class is actually contained in the continuous cohomology $H^2_c(H'',Z)$. Here $L'$ is equipped with the topology as a subgroup of the Schlichting completion $L=\Delta\sslash\Gamma$, and $Z$ is equipped with the subspace topology (and hence is discrete) of $L'$ and $H''$ is equipped with the quotient topology. To establish continuity, recall that a central extension corresponds to a continuous cohomology class if the extension is a product as topological spaces (Hu \cite[5.3]{Hu-cohomology}), i.e. if $L'$ is homeomorphic to $Z \times H''$.
\begin{claim} $L'$ is homeomorphic to $Z\times H''$. \label{claim:top-split} \end{claim}
\begin{proof} Since $H''\cong (N/C)/Z\times H'$, and $N/C$ is a discrete group and hence homeomorphic to $Z\times (N/C)/Z$, it suffices to show that $L'$ is homeomorphic to $(N/C)\times H'$.

By Claim \ref{claim:l-split} the short exact sequence
	$$1\to N/C\to L'\to H'\to 1$$
splits as an extension of topological groups when restricted to the compact open subgroup $\overline{\Lambda}^L/C$ of $L'$, i.e.
	$$1\to N/C\to \overline{\Lambda}^L/C\to \overline{\Lambda}^H\to 1$$
is split and this splitting is topological. Since $\overline{\Lambda}^L/C$ is open in $L'$, we can choose a (countable, discrete) set $\{l_i\}_{i\in I}$ of its coset representatives, and write
	$$L'=\bigsqcup_{i\in I} l_i \, \overline{\Lambda}^L/C.$$
Using $\overline{\Lambda}^L/C\cong (N/C)\times \overline{\Lambda}^H$ and that $N/C$ is normal in $L'$, we have a homeomorphism
	\begin{align*}
	L'	&\cong \bigsqcup_{i\in I} l_i ((N/C)\times \overline{\Lambda}^H) 	\\
		&= N/C\times \bigsqcup_i l_i \overline{\Lambda}^H				\\
		&=N/C \times H',
			\end{align*}
as desired.\end{proof}


This cohomological rephrasing is only useful if we can compute the relevant cohomology group $H^2_c(H'',Z)$:

\begin{claim} $H^2_c(H'',Z)$ is torsion.
\label{claim:h2-tor}
\end{claim}
\begin{proof} First note that $Z$ is finitely generated: Indeed, $Z$ has finite index in $N/C$ (Claim \ref{claim:z-nc-fi}) so it suffices to show $N/C$ is finitely generated. But $N/C$ is a quotient of $\Lambda/\Gamma$ (Claim \ref{claim:res-res-iso}) and hence is finitely generated because $\Lambda$ is. Therefore $Z$ is finitely generated and abelian, so that we can write $Z\cong \bbZ^r \oplus T$ for some finite abelian group $T$, and hence
	$$H^2_c(H'',Z)\cong H^2_c(H'',\bbZ)^r \oplus H^2_c(H'',T).$$
Since $T$ is finite, the second term is clearly torsion.
\begin{rem} For later use in the positive characteristic case (see Section \ref{sec:char+}), we note that we will not use that $T$ is finite, but merely that it has bounded exponent. \end{rem}
Hence it remains to show that $H^2_c(H'',\bbZ)$ is torsion. To do so, consider the short exact sequence of coefficients
	\begin{equation}
		0\to \bbZ\to \bbC\to \bbC/\bbZ\to 0.
	\label{eq:ses-coeff}
	\end{equation}	
Since this sequence is not split (as topological groups), this does not automatically yield a long exact sequence on continuous cohomology of (locally compact, second countable) topological groups. However, for totally disconnected groups, it is a result of Michael that there is a long exact sequence (see \cite[Thm M]{wigner-cohomology}). The part that is relevant for us is
	\begin{equation}
		\dots \to H^1_c(H'',\bbC/\bbZ)\to H^2_c(H'',\bbZ)\to H^2_c(H'',\bbC) \dots
		\label{eq:les-coeff}
	\end{equation}	
To describe $H^2_c(H'',\bbC)$, recall the following vanishing result for the continuous cohomology of semisimple groups:
\begin{thm}[{Casselman-Wigner \cite[Cor. 2]{casselman-wigner}}] Let $F$ be a non-archimedean local field. The $F$-rational points of a connected, semisimple, algebraic group over $F$ have vanishing continuous cohomology (with $\bbC$-coefficients) in positive degrees.\end{thm}
\begin{rem} The above citation is for the result with $F=\bbQ_p$. For the general case, see \cite[Rem. (2), p. 210]{casselman-wigner}.
\label{rem:any-field}\end{rem}
We will use this to show that $H^2_c(H'',\bbC)=0$:  Indeed, we have $H''\cong (N/C)/Z\times H'$ and $H'\subseteq H$ has finite index. The cohomology of $H''$ is then computed from the cohomology of $H'$ using the Hochschild-Serre spectral sequence (see \cite[Prop. 5]{casselman-wigner}), whose second page is $E_2^{pq}=H^p_c(H''/H', H^q_c(H',\bbC))$. Since $H''/H'\cong (N/C)/Z$ is finite, its cohomology vanishes in positive degrees so that the the second page is concentrated on $p=0$, where its values are $H^\bullet_c(H',\bbC)$. Further, $H'\subseteq H$ is a finite index normal subgroup and $H$, being the Schlichting completion of an $S$-arithmetic group with respect to its integral points, is a finite index, open, normal subgroup of a finite product of groups to which Casselman-Wigner's result applies (see Proposition \ref{prop:schlichting-arithmetic}), and the same spectral sequence argument shows that $H^2_c(H',\bbC)=0$, as desired.

Finally, to be able to use the long exact sequence \eqref{eq:les-coeff}, we consider $H^1_c(H'',\bbC/\bbZ)$: From the splitting $H''\cong (N/C)/Z\times H'$, it is clear that $H''$ has compact totally disconnected topological abelianization (i.e. the quotient by $\overline{[H'',H'']}$ is finite), and hence $H^1_c(H'',\bbC/\bbZ)$ is torsion.

From the long exact sequence \eqref{eq:les-coeff}, it follows that $H^2_c(H'',\bbZ)$ is also torsion, as desired.\end{proof}

\begin{rem} Again for use in the positive characteristic, we note we only used that the topological abelianization of $H''$ is compact and totally disconnected, and that $H$ is a closed, normal, cocompact subgroup of a group to which Casselman-Wigner's result applies with quotient that has bounded exponent. In characteristic zero, the proof in fact shows that $H^2_c(H'',Z)$ is finite. \end{rem}

%

So far we have found that $H^2_c(H'',Z)$ is torsion. We will now use this to show that $Z$ is finite. Let $\omega\in H^2_c(H'',Z)$ be the class corresponding to the central extension
	$$1\to Z\to L'\to H''\to 1,$$
and choose $m\geq 1$ such that $m\omega =0$. Then the extension $mL'$ corresponding to $m\omega$ is split. The natural map $L'\to mL'$ composed with the projection $mL'\to Z$ gives a continuous map $L'\to Z$. Further it restricts to multiplication by $m$ on $Z$. Hence if $Z$ is infinite, then the image of $L'\to Z$ is infinite. Since $\Delta'$ is dense in $L'$, the image of $\Delta$ in $Z$ is also infinite. This contradicts the normal subgroups theorem. We conclude that $Z$ is finite. \end{proof}

\subsection{End of the proof of Theorem \ref{thm:main}}\label{sec:end} In this section we prove:
\begin{thm} Let $\Delta$ be a finitely generated group with commensurated, finitely generated subgroup $\Lambda\subseteq \Delta$. Suppose $\Gamma$ is another commensurated subgroup of $\Delta$ that is normal in $\Lambda$. Further assume:
	\begin{enumerate}[(1)]
		\item $H^2_c(\Delta\sslash \Lambda, \bbC)=0,$ \label{assume:h2c=0}
		\item Any normal subgroup of $\Delta$ is finite or finite index. \label{assume:nst}
	\end{enumerate}
Then $\Gamma\subseteq\Lambda$ has finite index.
\label{thm:axiom}
\end{thm}
\begin{rems}\mbox{}
	\begin{enumerate}[(1)]
		\item Clearly, this implies Main Theorem \ref{thm:main}.
		\item Creutz-Shalom \cite{creutz-shalom} have given conditions on $\Delta\sslash\Lambda$ and $\Lambda\sslash\Gamma$ that guarantee Assumption \eqref{assume:nst}.
	\end{enumerate}
\end{rems}
By  Lemma \ref{lemma:c-fi}, $C$ has finite index in $N$ and hence is cocompact in $\oLL=N\oGL$. In the remainder of this section we will prove that $C$ is also compact, so that $\oLL$ is compact, and hence $\Gamma$ has finite index in $\Lambda$.

To prove compactness of $C :=\ll \hspace{-0.1 cm} \oGL\cap N\gg$ (where the normal closure is taken in $L$), we start by investigating the conjugation action of $L$ on $N$. The key property is that these automorphisms are inner modulo a compact subgroup:

\begin{lemma}  For every $\ell\in L$, let $c_\ell:N\to N$ denote conjugation by $\ell$. Then for every $\ell\in L$, there exists a compact open normal subgroup $C_\ell\subseteq N$ such that $\ell$ restricts to an automorphism of $C_\ell$ and descends to an inner automorphism of $N/C_\ell$.
\label{lemma:innermodcmpt}\end{lemma}
\begin{proof} For $\ell\in\oLL$, set $C_\ell:=\oGL\cap N$. Then it is clear that $c_\ell(C_\ell)=C_\ell$. We have $N/C_\ell \cong \Lambda/\Gamma\cong \oLL/\oGL$, and $c_\ell$ descends to conjugation of the image of $\ell$ in $\oLL/\oGL$. This proves the claim for $\ell\in\oLL$.

Hence the claim is also true for $\ell$ conjugate into $\oLL$. Finally, we note that $\oLL$ normally generates $L$ (since it contains $N$ and $\oLH$ normally generates $H=L/N$). So it remains to prove that the composition of automorphisms that are inner modulo compact is inner modulo compact. Let $\varphi,\psi$ be two automorphisms that are inner modulo $C_\varphi$ and $C_\psi$. Consider $C_{\varphi\psi}:=C_\varphi C_\psi = C_\psi C_\varphi$. Note that $C_{\varphi\psi}$ is a compact normal subgroup that is invariant under both $\varphi$ and $\psi$: For example, to see invariance under $\varphi$, note that $C_{\varphi\psi}/ C_\varphi$ is a compact normal subgroup of $N/C_\varphi$ and hence invariant under any inner automorphism.

Finally, note that $\varphi\circ \psi$ descends to an automorphism of $N/(C_{\varphi\psi})$ that is the composition of inner automorphisms, and hence is also inner.\end{proof}

Next, we show that in the above claim, there is a single compact subgroup $C_L$ of $N$ that works for all $\ell\in L$ simultaneously. This will immediately imply compactness of $C$:

\begin{lemma} There exists a compact subgroup $C_L\subseteq N$ containing $\oGL\cap N$ such that $C_L$ is normal in $L$. In particular, $C=\ll \hspace{-0.1 cm} \oGL\cap N\gg$ is compact.
\label{lemma:c-cmpt}\end{lemma}
\begin{proof} Let $S_\Delta$ be a finite generating set of $\Delta$. Since $\Delta\subseteq L$ is dense, it suffices to construct a compact subgroup $C_L\subseteq N$ containing $\oGL\cap N$ that is normalized by $S_\Delta$. For $s\in S_\Delta$, choose a compact normal subgroup $C_s\subseteq N$ such that $s$ is inner modulo $C_s$. Set
	$$C_L:=(\oGL\cap N)\prod_{s\in S_\Delta} C_s.$$
Then $C_L$ is compact and just as in the proof of the previous lemma, it is normalized by any $s\in S_\Delta$ because $C_L/C_s$ is invariant under any inner automorphism of $N/C_s$.\end{proof}

\section{Positive characteristic}
\label{sec:char+}

In this section, we will prove the analogous theorem in positive characteristic. The proof is largely identical, except where finite generation of lattices is used in Section \ref{sec:abquot}. The only issue that we cannot resolve in this generality is that if $\Lambda$ is non-uniform (and hence infinitely generated), we were not able to prove its arithmeticity. Therefore we have the following modified statement in positive characteristic:

\begin{thm} Let $k$ be a local field of positive characteristic $p>0$\, and $\bbG$ a connected semisimple algebraic group defined over $k$. Let $\Lambda$ be a lattice in $G:=\bbG(k)$ and suppose $\Gamma$ is an infinite normal subgroup of $\Lambda$ with commensurator $\Delta$ that is dense in $G$. Suppose that either $\Lambda$ is uniform, or that $\Lambda$ is arithmetic. Then $\Gamma$ is an arithmetic lattice and hence has finite index in $\Lambda$.
\label{thm:char+}
\end{thm}
\begin{proof} The proof is verbatim identical to the above proof in characteristic zero, except in the following places:

\subsection{Schlichting completions} This parallels Section \ref{sec:schlichting-def}. No changes are necessary.

\subsection{Arithmeticity of $\Lambda$} \label{sec:arithm-char+} This section used finite generation of $\Lambda$. However, in positive characteristic, nonuniform lattices in rank 1 groups are infinitely generated (see Lubotzky's Theorem \ref{thm:lubotzky-char+} below for a more precise statement), and therefore this part of the proof does not go through if $\Lambda$ is nonuniform. However, if $\Lambda$ is nonuniform, we assumed it is $T$-arithmetic (for some finite set of places $T$) so that we can simply skip this part of the proof.

As in Remark \ref{rem:one-prime}, we henceforth replace $\Delta$ by a finitely generated $S$-arithmetic subgroup for some finite set of places $S$ with $T\subsetneq S$. In particular satisfies the Normal Subgroups Theorem, and $\Delta\sslash\Lambda$ is given by the closure of $\Delta$ in $G_{S\bs T}^{\text{is}}$. This closure need no longer be finite index, but it is always a normal, cocompact subgroup with abelian, bounded exponent quotient.

\subsection{Maps of descent and restriction} No changes are necessary in \ref{sec:restrict}.

\subsection{Down to the central extension} No changes are necessary in \ref{sec:down}.

\subsection{Splitting of the central extension} Some modifications are required to establish that $H^2_c(H'',Z)$ is torsion (Claim \ref{claim:h2-tor}). Namely,  finite generation of $\Lambda$ is used to obtain finite generation of $Z$. However, as we mentioned above, in positive characteristic, nonuniform lattices in rank one groups are not finitely generated. We note that by using Margulis Normal Subgroups Theorem, we can reduce to the case where $G$ is rank one, non-compact and simple. In that case, we have the following result of Lubotzky that gives a `thick-thin decomposition' of lattices in such $G$:

\begin{thm}[{Lubotzky \cite[Thm. 2]{lubotzky-char+}}] Assume the notation of Theorem \ref{thm:char+}. Then $\Lambda$ contains a finite index subgroup $\Lambda'$ that can be written
	$$\Lambda' \cong F \ast \mathop{\Asterisk}_{1\leq j\leq d} U^j,$$
where $F$ is a finitely generated free group and $U^j$ are lattices in unipotent radicals of minimal parabolic subgroups of $G$.
\label{thm:lubotzky-char+}\end{thm}

Henceforth we fix $\Lambda', F, U^j$ as in the above theorem. Note that the groups $U^j$ are not finitely generated, and hence, neither is $\Lambda$. However, if $\bbG(k)\subseteq \SL(n,k)$, then any unipotent element of $\bbG(k)$ has order dividing $q:=p^n$: Indeed, $A^{p^n}-\Id = (A-\Id)^{p^n}=0$. Then we argue as follows to prove that $H^2_c(H'',Z)$ is torsion: Write $N/C=\Lambda/\Theta$ and let $S$ be the image in $N/C$ of finitely many coset representatives of $\Lambda/\Lambda'$ and the image of the generators of the free group $F$. Then $N/C$ is virtually abelian and generated by the finite set $S$ and the set $\sqrt[q]{\{e\}}$ of elements of exponent $q$.

Our first goal is to show that $Z$ is also generated by a set of this form. To see this, note that $Z\cap \sqrt[q]{\{e\}}$ is a (characteristic) subgroup of $Z$ (because $Z$ is abelian), and $\sqrt[q]{\{e\}}/(Z\cap \sqrt[q]{\{e\}})$ is finite (because $Z$ is central). Therefore $\Theta/(Z\cap \sqrt[q]{\{e\}})$ is finitely generated, and hence so is its finite index subgroup $Z/(Z\cap \sqrt[q]{\{e\}})$. Let $\overline{S}_Z$ be a finite set of generators of $Z/(Z\cap \sqrt[q]{\{e\}})$, and let $S_Z$ be a chosen set of pre-images of the elements of $\overline{S}_Z$. Then $S_Z\cup (Z\cap \sqrt[q]{\{e\}})$ generates $Z$, as desired.

Next, we will show that $Z\cong \bbZ^r\oplus T$ for some $r<\infty$ and bounded exponent group $T$. Indeed, since $Z/(\sqrt[q]{\{e\}}\cap Z)$ is finitely generated abelian, we see that $\sqrt[q]{\{e\}}\cap Z$ has finite index in the torsion subgroup $T:=\Tor(Z)$, and hence $T$ has bounded exponent. Further $Z/T$ is finitely generated and torsion-free, so it is isomorphic to $\bbZ^r$ for some $0\leq r<\infty$, as desired.

The above argument replaces the use of finite generation of $\Lambda$ in the proof of Claim \ref{claim:h2-tor}. The rest of the proof goes through: The vanishing of $H^2_c(H',\bbC)=0$ is established as before. To show that $H^1_c(H'',\bbC/\bbZ)\cong \Hom_c(H'',\bbC/\bbZ)$ is torsion, it suffices to show that the topological abelianization of $H''$ is compact (because it is also totally disconnected). By a result of Tits (see \cite{tits-simple} or \cite[Thm I.1.5.6(ii)]{margulis-book}), the commutator subgroup of $H$ is given by $H^+$, and compactness of $H/H^+$ is due to Borel-Tits (\cite[6.14]{borel-tits-morphismes-abstraits} or \cite[I.2.3.1(b)]{margulis-book}).

\subsection{End of the proof of Theorem \ref{thm:char+}} No changes are necessary because in this part of the proof, we do not use finite generation of $\Lambda$, only that  of $\Delta$. And while in positive characteristic, lattices need not be finitely generated, higher rank lattices always are (due to Venkataramana when one of the simple factors is higher rank \cite{venky-thesis} and Raghunathan if all simple factors have rank 1 \cite{raghunathan-fg}). In this case, $\Delta$ has higher rank since it is $S$-arithmetic and there exist at least two places in $S$ at which $\bbG$ is noncompact (one in $T$ and one in $S\bs T$). \end{proof}


\section{Supergroups of arithmetic lattices}
\label{sec:venky}

In this section we prove the technical Lemma \ref{lem:gen-venky} on subgroups of $S$-arithmetic lattices containing an arithmetic lattice. With minor additional assumptions, this is due to Venkataramana (see \cite[Prop. 2.3]{lubotzky-zimmer}) with essentially the same proof. This fact is probably well-known to experts but we have not been able to locate a version in full generality in the literature. For convenience we restate the result:
\begin{lemma} Let $\bbG$ be a connected, almost $k$-simple algebraic group defined over a global field $k$. Let $T$ be a finite set of places of $k$ (containing all archimedean places if char$(k)=0$), and let $\Lambda$ be $T$-arithmetic.

Let $\Theta\subseteq \bbG(k)$ be a subgroup containing $\Lambda$ whose projection to $\bbG(k_s)$ is bounded for almost all places $s$, and let $S$ be the (finite) set of places where the projection of $\Theta$ is unbounded. Then $\Theta$ is $S$-arithmetic.
\label{lem:regen-venky}\end{lemma}
\begin{proof} Set $G_s:=\bbG(k_s)$ and $G_S:=\prod_{s\in S} G_s$. Consider the universal cover $\pi_S:\widetilde{G}_S\to G_S$. Since the image under $\pi_T$ (resp. $\pi_S$) of a $T$-arithmetic (resp. $S$-arithmetic) subgroup of $\widetilde{\bbG}(k)$ is $T$-arithmetic (resp. $S$-arithmetic) in $\bbG(k)$ (see e.g. \cite[I.3.2.9]{margulis-book}), we can replace $\Lambda$ and $\Theta$ by their intersections with the image of $\pi_S$.

We argue by induction on $|S\bs T|$, so first suppose $S\bs T=\{s\}$. Let $K_s:=\overline{\Lambda}$ be the closure of $\Lambda$ in $G_s$. Then $K_s$ is compact and open in  $\pi_s(\widetilde{G}_s)$. Here, openness follows from applying strong approximation in $\widetilde{\bbG}$ (due to Platonov in characteristic zero \cite{platonov-strongapprox}, Prasad in positive characteristic \cite{prasad-strongapprox}, and see \cite[Proposition 7.2(2)]{platonov-rapinchuk-book} for a reformulation in terms of $S$-integers), which applies because $G_s$ is noncompact (since $\Theta$ has unbounded projection to $G_s$).



Let $G_s^{\text{is}}$ be the quotient of $G_s$ by its compact factors. We claim that the image of $\Theta$ under the diagonal embedding $\diag:\Theta\to G_T \times (G_s^{\text{is}})^+$ is a lattice in $G_T\times (G_s^{\text{is}})^+$. Let $F\subseteq G_T$ be a $\Lambda$-fundamental domain of finite measure, and let $K_s^{\text{is}}$ be the image of $K_s\subseteq G_s$ in $G_s^{\text{is}}$. We will show that
	$$F\times K_s^{\text{is}}\subseteq G_T\times (G_s^{\text{is}})^+$$
is a $\diag(\Theta)$-fundamental domain for $\diag(\Theta)\subseteq G_T\times (G_s^{\text{is}})^+$. First, we argue the $\diag(\Theta)$-translates of $F\times K_s^{\text{is}}$ are disjoint. Suppose that $\theta\in\diag(\Theta)$ is such that $\theta(F\cap K_s^{\text{is}}) \cap (F\times  K_s^{\text{is}})\neq \emptyset$. By inspecting the second factor and using that $K_s^{\text{is}}$ is a subgroup, we obtain that the projection of $\theta$ to the second factor lies in $K_s^{\text{is}}$. Since $\Theta\cap K_s=\Lambda$ and $\Theta\subseteq G_s$ does not intersect the product of compact factors of $G_s$, we have $\theta\in\diag(\Lambda)$. Then since $F$ is a fundamental domain for $\Lambda\subseteq G$ and $\theta F\cap F\neq\emptyset$, we find that $\theta$ is trivial.

Now we show $\diag(\Theta)(F\times K_s^{\text{is}})=G_T\times (G_s^{\text{is}})^+$. Since $K_s^{\text{is}}$ is $\Lambda$-invariant and $F$ is a fundamental domain for $\Lambda\subseteq G$, it suffices to show that $\Theta K_s^{\text{is}}= (G_s^{\text{is}})^+$. Since in addition $K_s^{\text{is}}\subseteq (G_s^{\text{is}})^+$ is open, it suffices to show $\Theta$ is dense in $(G_s^{\text{is}})^+$.

But the pre-image $\widetilde{H}_s$ of $H_s:=\overline{\Theta}$ (closure taken in $\pi_s(\widetilde{G}_s)$) in the universal cover $\widetilde{G}_s$ is open (by strong approximation) and unbounded. However, any open, noncompact subgroup of $\widetilde{G}_s$ contains the subgroup $\widetilde{G}_s^+$ (an unpublished result of Tits with a published proof due to Prasad \cite{prasad-tits}) and the image of $\widetilde{G}_s^+$ in $G_s$ is precisely $G_s^+$ (Borel-Tits, see Proposition \ref{prop:plus-isogeny}), which surjects onto $(G_s^{\text{is}})^+$ . So we conclude that $\Theta$ is dense in $(G_s^{\text{is}})^+$. This completes the proof in the base case that $|S\bs T|=1$.

Now assume that for some $r\geq 1$, the result is true whenever $|S\bs T|\leq r$, and suppose that $|S\bs T|=r+1$. Fix some subset $T\subseteq T_1\subseteq S$ such that $S\bs T_1 = \{s\}$ consists of a single place, let $K_s$ be a maximal compact subgroup of $G_s$ containing $\Lambda$, and set $\Theta_s:=\Theta\cap K_s$. Then $\Theta_s$ contains $\Lambda$. Let $T_2\subseteq T_1$ denote the set of places at which the projection of $\Theta_s$ is unbounded. By the inductive hypothesis, $\Theta_s$ is $T_2$-arithmetic. By the base case (applied with $\Lambda$ replaced by $\Theta_s$), we find that $\Theta$ is then $(T_2\cup\{s\})$-arithmetic.

Finally, note that $T_2\cup\{s\}=S$ because by assumption $S$ consists exactly of those places at which $\Theta$ has unbounded projection, and a $(T_2\cup \{s\})$-arithmetic lattice has unbounded projection precisely to the places in $T_2\cup\{s\}$.\end{proof}

\section{The reducible case} \label{sec:reducible}
In this section, we prove a version of the main theorem for reducible lattices.  For simplicity, we state only the result in characteristic zero. The interested reader can combine the proof here with the partial result in Theorem \ref{thm:char+} to obtain the most general result.

\begin{thm} Let $G_i$, $i\in I$, be finitely many semisimple algebraic groups defined over local fields of characteristic zero. Set $G:=\prod_i G_i$ and let $\Lambda\subseteq G$ be a lattice such that the projection of $\Lambda$ to every factor $G_i$ is an irreducible lattice $\Lambda_i$.

Suppose $\Gamma\subseteq \Lambda$ is a normal subgroup with dense commensurator. Then there is a subset $J\subseteq I$ of factors such that $\Gamma$ has finite index in $\prod_{j\in J} \Lambda_j$.
\label{thm:reducible}
\end{thm}
\begin{proof} Note that $\Lambda\subseteq \prod_{i\in I} \Lambda_i$ has finite index (since $\Lambda$ is a lattice and $\prod_i \Lambda_i$ is discrete). Then there is a finite index subgroup of $\Gamma$ that is normalized by $\prod_{i\in I} \Lambda_i$, and hence we can replace $\Lambda$ by $\prod_{i\in I} \Lambda_i$.

Let $p_i:G\to G_i$ denote the projection onto a factor. Then $p_i(\Gamma)$ is a normal subgroup of $\Lambda_i$ that is commensurated by $p_i(\Delta)$, and hence by the Main Theorem \ref{thm:main}, $p_i(\Gamma)$ is either finite or has finite index in $\Lambda_i$.

Let $J$ denote the subset of factors such that $p_i(\Gamma)\subseteq\Lambda_i$ has finite index, and set $G_J=\prod_{j\in J} G_j$. Then $\Gamma\cap G_J$ has finite index in $\Gamma$, and hence a further finite index subgroup $\Gamma'$ is contained in $G_J$ and also normal in $\Lambda$. By replacing $G$ by $G_J$ and $\Gamma$ by $\Gamma'$, we may assume that $p_i(\Gamma)\subseteq \Lambda_i$ has finite index for all $i$. By replacing $\Lambda$ by its finite index subgroup $\prod_i p_i(\Gamma)$, we may assume $p_i(\Gamma)=\Lambda_i$.

We will now show that $\Gamma\subseteq \prod_i \Lambda_i$ has finite index. Indeed, for any $i\in I$, the intersection with a factor $N_i:=\Gamma\cap G_i$ is a normal subgroup of $\Lambda_i$, and it is easy to see $p_i(\Delta)$ commensurates $N_i$: Indeed for $\delta\in\Delta$, the natural map
	$$N_i/(N_i\cap N_i^\delta) \to \Gamma / (\Gamma\cap\Gamma^\delta)$$
is injective because
	\begin{align*}
	N_i\cap (\Gamma \cap \Gamma^\delta) 	&= G_i \cap \Gamma\cap \Gamma^\delta \\
										&= G_i\cap \Gamma \cap (G_i\cap \Gamma)^\delta \\
										&= N_i\cap N_i^\delta.
	\end{align*}										
Therefore by the Main Theorem \ref{thm:main} applied to $N_i$, we obtain that $N_i\subseteq \Lambda_i$ is either finite or finite index. We will argue by contradiction that $N_i$ must be infinite for all $i$, which will complete the proof.

Suppose then that there is $i\in I$ such that $N_i$ is finite. Since $p_i(\Gamma)=\Lambda_i$ is residually finite, there is a finite index subgroup $\Gamma'\subseteq \Gamma$ such that $\Gamma'\cap N_i=1$. By passing to a further finite index subgroup, we can assume that $\Gamma'\subseteq\Lambda$ is normal, and henceforth we replace $\Gamma$ by $\Gamma'$, so that $N_i=1$. It follows that the projection $q_i:G\to \prod_{j\neq i} G_j$ away from $G_i$ restricts to an isomorphism on $\Gamma$, and hence $\Gamma$ is the graph of a surjective homomorphism $q_i(\Gamma)\to \Lambda_i$. But the graph of a homomorphism $\varphi:A\to B$ is not normal in $A\times B$ unless its image is contained in the center of $B$. Since $\Lambda_i$ is not abelian, this is a contradiction.\end{proof}

\bibliographystyle{alpha}
\bibliography{ref}
\end{document}